\documentclass[12pt,british]{article}
\usepackage{amsmath,amsfonts,amssymb,amscd,amsthm}
\usepackage{babel}

\usepackage{ucs} 
\usepackage[utf8x]{inputenc} 
\usepackage[T1]{fontenc} 

\usepackage[noBBpl]{mathpazo}
\usepackage{textcomp}
\usepackage{eucal}
\usepackage{enumerate}
\usepackage{mathabx}

\usepackage[matrix,arrow]{xy}

\renewcommand{\epsilon}{\ensuremath{\varepsilon}}
\renewcommand{\to}{\ensuremath{\longrightarrow}}
\newcommand{\rp}[1][2]{\ensuremath{{\mathbb R}P^{#1}}}

\newcommand{\N}{\ensuremath{\mathbb N}}
\newcommand{\Z}{\ensuremath{\mathbb Z}}
\newcommand{\dt}{\ensuremath{\mathbb D}^{\,2}}
\newcommand{\St}[1][2]{\ensuremath{\mathbb S}^{#1}}

\newcommand{\sn}[1][n]{\ensuremath{S_{{#1}}}}
\renewcommand{\ker}[1]{\ensuremath{\operatorname{\text{Ker}}({#1})}}

\makeatletter
\def\@enum@{\list{\csname label\@enumctr\endcsname}%
           {\usecounter{\@enumctr}\def\makelabel##1{
\normalfont\ignorespaces\emph{{##1}~}}
\setlength{\labelsep}{3pt}
\setlength{\parsep}{0pt}
\setlength{\itemsep}{0pt}
\setlength{\leftmargin}{0pt}
\setlength{\labelwidth}{0pt}
\setlength{\listparindent}{\parindent}%
\setlength{\itemsep}{0pt}
\setlength{\itemindent}{0pt}
\setlength{\topsep}{3pt plus 1pt minus 1 pt}}}
\renewcommand\theenumi{\@alph\c@enumi}
\renewcommand\theenumii{\@alph\c@enumii}
\renewcommand\theenumiii{\@alph\c@enumiii}
\renewcommand\theenumiv{\@alph\c@enumiv}

\makeatother

\makeatletter
\def\@map#1#2[#3]{\mbox{$#1 \colon\thinspace #2 \to #3$}}
\def\map#1#2{\@ifnextchar [{\@map{#1}{#2}}{\@map{#1}{#2}[#2]}}
\DeclareRobustCommand*\textsubscript[1]{\@textsubscript{\selectfont#1}}
\def\@textsubscript#1{{\m@th\ensuremath{_{\mbox{\fontsize\sf@size\z@#1}}}}}
\makeatother

\newcommand{\id}{\ensuremath{\operatorname{\text{Id}}}}

\DeclareRobustCommand*{\up}[1]{\textsuperscript{#1}}
\renewcommand{\th}{\ensuremath{\up{th}}}

\newcommand{\garside}[1][n]{\ensuremath{\Delta_{#1}}}
\newcommand{\ft}[1][n]{\ensuremath{\Delta^2_{#1}}}

\newcommand{\brak}[1]{\ensuremath{\left\{ #1 \right\}}}
\newcommand{\ang}[1]{\ensuremath{\left\langle #1\right\rangle}}

\newcommand{\setangr}[2]{\ensuremath{\ang{#1 \,\left\lvert \, #2 \right.}}}

\newcommand{\setr}[2]{\ensuremath{\brak{#1 \,\left\lvert \, #2 \right.}}}
\newcommand{\setl}[2]{\ensuremath{\brak{\left. #1 \,\right\rvert \, #2}}}

\newcommand{\an}[1][n]{\ensuremath{A_{{#1}}}}

\newcommand{\dih}[1]{\ensuremath{\operatorname{\text{Dih}}_{#1}}}
\newcommand{\dic}[1]{\ensuremath{\operatorname{\text{Dic}}_{#1}}}
\newcommand{\quat}[1][8]{\ensuremath{\mathcal{Q}_{#1}}}

\newcommand{\tonestar}{\ensuremath{\operatorname{T}^{\ast}}}
\newcommand{\oonestar}{\ensuremath{\operatorname{O}^{\ast}}}
\newcommand{\istar}{\ensuremath{\operatorname{I}^{\ast}}}

\newcommand{\reth}[1]{Theorem~\protect\ref{th:#1}}
\newcommand{\relem}[1]{Lemma~\protect\ref{lem:#1}}
\newcommand{\repr}[1]{Proposition~\protect\ref{prop:#1}}
\newcommand{\reco}[1]{Corollary~\protect\ref{cor:#1}}
\newcommand{\resec}[1]{Section~\protect\ref{sec:#1}}
\newcommand{\req}[1]{equation~(\protect\ref{eq:#1})}
\newcommand{\reqref}[1]{(\protect\ref{eq:#1})}

\newtheoremstyle{theoremm}{}{}{\itshape}{}{\scshape}{.}{ }{}

\theoremstyle{theoremm}
\newtheorem{thm}{Theorem}
\newtheorem{lem}[thm]{Lemma}
\newtheorem{prop}[thm]{Proposition}
\newtheorem{cor}[thm]{Corollary}

\newtheoremstyle{remarkk}{}{}{}{}{\scshape}{.}{ }{}

\theoremstyle{remarkk}

\newtheorem{rem}[thm]{Remark}

\newcommand{\sii}[2][1]{\ensuremath{\sigma_{#2}^{-{#1}}}}

\newcommand{\si}[2][{}]{\ensuremath{\sigma_{#2}^{#1}}}
\newcommand{\rh}[2][{}]{\ensuremath{\rho_{#2}^{#1}}}
\newcommand{\rhi}[2][1]{\ensuremath{\rho_{#2}^{-{#1}}}}

\newcommand{\ssni}[2]{\ensuremath{\sigma_{#1}^{-1}\cdots\sigma_{#2}^{-1}}}

\begin{document}

\title{Embeddings of the braid groups of covering spaces, classification of the finite subgroups of the braid groups of the real projective plane, and linearity of braid groups of low-genus surfaces}

\author{DACIBERG~LIMA~GON\c{C}ALVES\\
Departamento de Matem\'atica - IME-USP,\\
Caixa Postal~66281~-~Ag.~Cidade de S\~ao Paulo,\\ 
CEP:~ 05314-970 - S\~ao Paulo - SP - Brazil.\\
e-mail:~\texttt{dlgoncal@ime.usp.br}\vspace*{4mm}\\
JOHN~GUASCHI\\
Laboratoire de Math\'ematiques Nicolas Oresme UMR CNRS~\textup{6139},\\
Universit\'e de Caen BP 5186,
14032 Caen Cedex, France,\\
and\\
Instituto de Matem\'aticas, UNAM,
Le\'on \#2, altos, col.\ centro,\\
Oaxaca de Ju\'arez, Oaxaca,
C.P. 68000, Mexico.\\
e-mail:~\texttt{guaschi@math.unicaen.fr}}

\date{9th June 2009}

\begingroup
\renewcommand{\thefootnote}{}
\endgroup 



\maketitle

\vspace*{-4mm}

\begin{abstract}
Let $M$ be a compact, connected surface, possibly with a finite set of points removed from its interior. Let $d,n\in \N$, and let $\widetilde{M}$ be a $d$-fold covering space of $M$.  We show that the covering map induces an embedding of the $n\up{th}$ braid group $B_{n}(M)$ of $M$ in the $dn\up{th}$ braid group $B_{dn}(\widetilde{M})$ of $\widetilde{M}$, and give several applications of this result. First, we classify the finite subgroups of the $n\up{th}$ braid group of the real projective plane, from which we deduce an alternative proof of the classification of the finite subgroups of the mapping class group of the $n$-punctured real projective plane due to Bujalance, Cirre and Gamboa. Secondly, using the linearity 
of $B_{n}$ due to Bigelow and Krammer, we show that the braid groups of compact, connected surfaces of low genus are linear.
\end{abstract}


\section{Introduction}

The braid groups $B_n$ of the plane were introduced by E.~Artin in~1925~\cite{A1,A2}. Braid groups of surfaces were studied by Zariski~\cite{Z}. They were later generalised by Fox to braid groups of arbitrary topological spaces via the following definition~\cite{FoN}. Let $M$ be a compact, connected surface, and
let $n\in\N$. We denote the set of all ordered $n$-tuples of distinct points of $M$, known as the \emph{$n\th$ configuration space of $M$}, by:
\begin{equation*}
F_n(M)=\setr{(p_1,\ldots,p_n)}{\text{$p_i\in M$ and $p_i\neq p_j$ if $i\neq j$}}.
\end{equation*}
Configuration spaces play an important r\^ole in several branches of mathematics and have been
extensively studied, see~\cite{CG,FH} for example. 

The symmetric group $\sn$ on $n$ letters acts freely on $F_n(M)$ by permuting coordinates. The corresponding quotient will be denoted by $D_n(M)$, and will be termed the \emph{$n\th$ permuted configuration space of $M$}. The \emph{$n\th$ pure braid group $P_n(M)$} (respectively the \emph{$n\th$ braid group $B_n(M)$}) is defined to be the fundamental group of $F_n(M)$ (respectively of $D_n(M)$). If $N$ is a subsurface of $M$ and $m\geq 0$, Paris and Rolfsen study the homomorphism $B_{n}(N)\to B_{n+m}(M)$ of braid groups induced by inclusion of $N$ in $M$, and give necessary and sufficient conditions for it to be injective~\cite{PR}. In this paper, we study the relationship between the braid groups of covering spaces. In \resec{inject}, we prove the following result which we hope will be useful in the understanding of various aspects of surface braid groups, especially the relationship between the braid groups of a non-orientable surface and those of its orientable double covering.

\begin{thm}\label{th:covering}
Let $M$ be a compact, connected surface, possibly with a finite set of points removed from its interior. Let $d,n\in \N$, and let $\widetilde{M}$ be a $d$-fold covering space of $M$.  Then the covering map induces an embedding of the $n\up{th}$ braid group $B_{n}(M)$ of $M$ in the $dn\up{th}$ braid group $B_{dn}(\widetilde{M})$ of $\widetilde{M}$.
\end{thm}

We have the following immediate consequence of \reth{covering}:
\begin{cor}\label{cor:embed}
Let $n\in \N$. The $n\up{th}$ braid group of a non-orientable surface embeds in the $2n\up{th}$ braid group of its orientable double cover. In particular, $B_{n}(\rp)$ embeds in $B_{2n}(\St)$.
\end{cor}
Note however that if $n\geq 2$, such an embedding does not restrict to an embedding of the corresponding pure braid subgroups since $P_{n}(\rp)$ has torsion $4$ (see \repr{agt}(\ref{it:jgt})) and $P_{2n}(\St)$ has torsion $2$.

Together with the $2$-sphere $\St$, the braid groups of the real projective plane $\rp$ are of particular interest, notably because they have non-trivial centre~\cite{VB,GG2}, and torsion elements~\cite{VB,M}. Indeed, Van Buskirk showed that $\St$ and $\rp$ are the only surfaces whose braid groups have torsion~\cite{VB}. Let us recall briefly some properties of $B_n(\rp)$.

If $\dt\subseteq \rp$ is a topological disc, there is a group homomorphism $\map
{\iota}{B_n}[B_n(\rp)]$ induced by the inclusion. If $\beta\in B_n$ then we shall denote its image
$\iota(\beta)$ simply by $\beta$. In \repr{present}, we recall a presentation of $B_{n}(\rp)$ due to Van Buskirk~\cite{VB}. In~\cite{GG4}, a presentation of $P_{n}(\rp)$ was given, and the splitting problem for the Fadell-Neuwirth short exact sequence of its pure braid groups was solved. The first two braid groups of $\rp$ are finite: $B_1(\rp)=P_1(\rp)\cong \Z_{2}$, $P_{2}(\rp)$ is isomorphic to the quaternion group $\quat$ of order~$8$, and $B_{2}(\rp)$ is isomorphic to the generalised quaternion group of order $16$. For $n\geq 3$, $B_{n}(\rp)$ is infinite. The pure braid group $P_{3}(\rp)$ is isomorphic to a semi-direct product of a free group of rank $2$ by $\quat$~\cite{VB}; an explicit action was given in~\cite{GG2}.

The so-called `full twist' braid of $B_n(\rp)$ is defined by $\ft= (\sigma_1\cdots\sigma_{n-1})^n$, where $\sigma_1,\ldots,\sigma_{n-1}$ are the standard generators of $B_{n}$, and is the square of the `half twist' braid $\garside$ defined by
\begin{equation}\label{eq:garside}
\garside= (\sigma_{1}\cdots \sigma_{n-1})(\sigma_{1}\cdots \sigma_{n-2}) \cdots (\sigma_{1} \sigma_{2})\sigma_{1}.
\end{equation}
For $n\geq 2$, Murasugi showed that $\ft$ generates the centre of $B_n(\rp)$, and he characterised the finite order elements of $B_{n}(\rp)$~\cite{M}, although their orders are not clear, even for elements of $P_{n}(\rp)$. In~\cite{GG2,GG9}, we proved the following results that include the classification of the finite subgroups of $P_{n}(\rp)$:
\begin{prop}[\cite{GG2,GG9}]\label{prop:agt}
Let $n\geq 2$. Then:
\begin{enumerate}[(a)]
\item\label{it:agt1} $B_{n}(\rp)$ has an element of order $\ell$ if and only if $\ell$ divides either $4n$ or
$4(n-1)$.
\item the (non-trivial) torsion of $P_{n}(\rp)$ is precisely $2$ and $4$.
\item\label{it:agt2} the full twist $\ft$ is the unique element of $B_{n}(\rp)$ of order $2$.
\item\label{it:jgt} up to isomorphism, the maximal finite subgroups of $P_{n}(\rp)$ are $\quat$ if $n=2,3$, and $\Z_{4}$ if $n\geq 4$.
\item  if $n\geq 3$, then up to isomorphism, the infinite virtually cyclic subgroups of $P_{n}(\rp)$ are $\Z$,
$\Z_{2}\times \Z$ and $\Z_{4} \ast_{\Z_{2}} \Z_{4}$.
\end{enumerate}
\end{prop}

For $m\geq 2$, let \dic{4m} denote the dicyclic group of order $4m$. It has the following presentation:
\begin{equation*}
\dic{4m}=\setangr{x,y}{x^m=y^2,\; yxy^{-1}=x^{-1}}.
\end{equation*}
The classification of the finite subgroups of $B_{n}(\St)$ and $B_{n}(\rp)$ is an interesting
problem, and helps us to better understand their group structure. In the case of $\St$, this was
undertaken in~\cite{GG5,GG6}:

\begin{thm}[\cite{GG6}]\label{th:finitebn}
Let $n\geq 3$. The maximal finite subgroups of $B_n(\St)$ are:
\begin{enumerate}
\item\label{it:fina} $\Z_{2(n-1)}$ if $n\geq 5$.
\item\label{it:finb} $\dic{4n}$.
\item\label{it:finc} $\dic{4(n-2)}$ if $n=5$ or $n\geq 7$.
\item\label{it:find} the binary tetrahedral group, denoted by $\tonestar$, if $n\equiv 4 \pmod 6$.
\item\label{it:fine} the binary octahedral group, denoted by $\oonestar$, if $n\equiv 0,2\pmod 6$.
\item\label{it:finf} the binary icosahedral group, denoted by $\istar$, if $n\equiv 0,2,12,20\pmod{30}$.
\end{enumerate}
\end{thm}
More information concerning the binary polyhedral groups $\tonestar, \oonestar$ and $\istar$ may be found in~\cite{AM}. The result of \reth{finitebn} leads us to ask which finite groups are realised as subgroups of $B_n(\rp)$. As in the case of $B_{n}(\St)$, one common property of such subgroups is that they are finite periodic groups of cohomological period $2$ or $4$. Indeed, by \cite[Proposition~6]{GG2}, for all $n\geq 2$, the universal covering $X$ of $F_n(\rp)$ is a finite-dimensional complex which has the homotopy type of $\St[3]$. Thus any finite subgroup of $B_n(\rp)$ acts freely on $X$, and so has period $2$ or $4$ by \cite[Proposition~10.2, Section~10, Chapter~VII]{Br}. Since $\ft$ is the unique element of order~$2$ of $B_n(\rp)$, and it generates the centre of $B_n(\rp)$, the Milnor property must be satisfied for any finite subgroup of $B_n(\rp)$. As an application of \reco{embed}, we classifiy the finite subgroups of $B_{n}(\rp)$.

\begin{thm}\label{th:finitebnrp2}
Let $n\geq 2$. The maximal finite subgroups of $B_n(\rp)$ are isomorphic to the following groups:
\begin{enumerate}
\item\label{it:finrpa} $\dic{8n}$.
\item\label{it:finrpb} $\dic{8(n-1)}$ if $n\geq 3$.
\item\label{it:finrpc} $\oonestar$ if $n\equiv 0,1\pmod 3$.
\item\label{it:finrpd} $\istar$ if $n\equiv 0,1,6,10\pmod{15}$.
\end{enumerate}
\end{thm}

The proof of \reth{finitebnrp2} is given in \resec{classfinite}, and is obtained by combining \reco{embed} with \reth{finitebn}. In this way, we obtain a list of possible finite subgroups of $B_n(\rp)$ in \resec{proofa}. Some of these possibilities are not realised (notably $\tonestar$ is not realised if $n\equiv 2 \pmod 3$, despite apparently being compatible with the embedding). The final step is to prove that the subgroups given in the statement of \reth{finitebnrp2} are indeed realised for the given values of $n$. This is achieved in \resec{proofb} using the geometric constructions of~\cite{GG6} of the finite subgroups of $B_{n}(\St)$, as well as \reco{embed} and the following short exact sequence due to Scott~\cite{Sc}, which exists if $n\geq 2$: 
\begin{equation}\label{eq:sesrp2}
1 \to \ang{\ft} \to B_{n}(\rp)\to \operatorname{\mathcal{MCG}}(\rp,n) \to 1,
\end{equation}
where $\operatorname{\mathcal{MCG}}(\rp,n)$ denotes the \emph{mapping class group} of the projective plane relative to the $n$-point subset $X$ that consists of the basepoints of the strands of $B_{n}(\rp)$, that is, the set of isotopy classes of homeomorphisms of $\rp$ that leave $X$ invariant (we allow the points of $X$ to be permuted). In \resec{realdicyc}, we give explicit algebraic realisations of the dicyclic subgroups of $B_n(\rp)$ in terms of Van Buskirk's presentation of $B_{n}(\rp)$~\cite{VB}.

The finite subgroups of $\operatorname{\mathcal{MCG}}(\rp,n)$ were classified by Bujalance, Cirre and Gamboa~\cite{BCG}. In \resec{mcg}, we use \reth{finitebnrp2} and \req{sesrp2} to obtain an alternative proof of their result.

\enlargethispage{4mm}

\begin{thm}[\cite{BCG}]\label{th:mcg}
Let $n\geq 2$. The maximal finite subgroups of $\operatorname{\mathcal{MCG}}(\rp,n)$ are isomorphic to the following groups:
\begin{enumerate}
\item the dihedral group $\dih{4n}$ of order $4n$.
\item the dihedral group $\dih{4(n-1)}$ if $n\geq 3$.
\item $\sn[4]$ if $n\equiv 0,1\pmod 3$.
\item $\an[5]$ if $n\equiv 0,1,6,10\pmod{15}$.
\end{enumerate}
\end{thm}
 
Another application of \reth{covering} concerns the linearity of certain surface braid groups. Recall that a group is said to be \emph{linear} if it admits a faithful representation in a multiplicative group of matrices over some field. Krammer~\cite{Kr1,Kr2} and Bigelow~\cite{Big} showed that $B_{n}$ is linear. The linearity of $\operatorname{\mathcal{MCG}}(\St,n)$ was proved in~\cite{Bar,BB,Kor}, and that of $B_{n}(\St)$ was obtained in~\cite{Bar,BB}. If $n\leq 2$ then $B_{n}(\rp)$ is linear because it is finite, while $B_{3}(\rp)$ is known to be isomorphic to a subgroup of $\operatorname{\mathrm{GL}}(96,\Z)$~\cite{Bar}. To the best of our knowledge, nothing is known about the linearity of $n\up{th}$ braid groups, $n\geq 2$, for other surfaces (in the case $n=1$, surface groups are known to be linear).  In \resec{linear} we give a short alternative proof in \repr{linear} of the linearity of $B_{n}(\St)$, and with the aid of \reco{embed}, we deduce the following results.
\begin{thm}\label{th:linear}
Let $n\in \N$.
\begin{enumerate}[(a)]
\item Let $M$ be a compact, connected surface, possibly with boundary, of genus zero if $M$ is orientable, and of genus one if $M$ is non-orientable. Then $B_{n}(M)$ is linear. 
\item Let $\mathbb{T}^2$ denote the $2$-torus, and let $x\in \mathbb{T}^2$. Then $B_{n+1}(\mathbb{T}^2)$ is linear if and only if $B_{n}(\mathbb{T}^2\setminus\brak{x})$ is linear. Consequently, $B_{2}(\mathbb{T}^2)$ is linear.
\end{enumerate}
\end{thm}

In particular, the braid groups of $\rp$ and the M\"obius band are linear.





\subsection*{Acknowledgements}

This work took place during the visit of the second author to the
Departmento de Mate\-m\'atica do IME~--~Universidade de S\~ao Paulo during
the periods~14\up{th}~--~29\up{th}~April~2008, 18\up{th}~July~--~8\up{th} August~2008, 31\up{st}~October~--~10\up{th}~November~2008 and 20\up{th}~May~--~3\up{rd}~June 2009, and of the visit of the first author to the Laboratoire de Math\'ematiques Nicolas Oresme, Universit\'e de Caen during the period 21\up{st}~November~--~21\up{st}~December~2008. This work was supported by the international Cooperation USP/Cofecub project n\up{o} 105/06, by the CNRS/CNPq project n\up{o}~21119 and by the~ANR project TheoGar n\up{o} ANR-08-BLAN-0269-02. The writing of part of this paper took place while  the second author was at the Instituto de Matem\'aticas, UNAM Oaxaca, Mexico. He would like to thank the CNRS for having granted him a `délégation' during this period, CONACYT for partial financial support through its programme `Estancias postdoctorales y sab\'aticas vinculadas al fortalecimiento de la calidad del posgrado nacional', and the Instituto de Matem\'aticas for its hospitality and excellent working atmosphere.

\section{Embeddings of the braid groups of coverings}\label{sec:inject}

Let $N$ be a compact, connected surface, let $X$ be an $r$-point subset in the interior of $N$, where $r\geq 0$, and let $M=N\setminus X$. Let $d,n\in \N$, and let $\map{p}{\widetilde{M}}[M]$ be a $d$-fold covering map of $M$. The aim of this section is to prove \reth{covering}. We use the covering map $p$ to construct a map $\map{\psi_{n}}{D_{n}(M)}[D_{dn}(\widetilde{M})]$ of configuration spaces, and then go on to show that $\psi_{n}$ induces an injective homomorphism $\map{\psi_{n\#}}{B_{n}(M)}[B_{dn}(\widetilde{M})]$. This homomorphism may be used to compare the braid groups of $M$ with those of $\widetilde{M}$, and its existence seems like an interesting fact in its own right.

Let $z_1,\ldots,z_n$ be $n$ base points lying in a small disc in the interior of $M$. Given a subset $A$ of $n$ distinct unordered points in $M$ (an element of $D_n(M)$ in other words), $p^{-1}(A)$ is a subset of $\widetilde{M}$ consisting of $dn$ distinct unordered points, and so is an element of $D_{dn}(\widetilde{M})$. The correspondence $A \longmapsto p^{-1}(A)$ thus defines a continuous map 
$\map{\psi_{n}}{D_n(M)}[D_{dn}(\widetilde{M})]$
of permuted configuration spaces whose base points are the sets $\brak{z_{1},\ldots,z_{n}}$ and $\brak{p^{-1}(z_1),\ldots,p^{-1}(z_n)}$ respectively. So the map $\psi_{n}$ induces a homomorphism $\map{\psi_{n\#}}{B_n(M)}[B_{dn}(\widetilde{M})]$ on the level of the fundamental groups (with the given base points).

We now prove that $\psi_{n\#}$ is injective by induction on $n$. We first deal with the case $n=1$.
 
\begin{lem}\label{lem:n1inj}
The homomorphism $\map{\psi_{1\#}}{B_1(M)}[B_{d}(\widetilde{M})]$ is injective.  Furthermore, for all $q\geq 2$, the induced homomorphisms $\pi_q(M) \to \pi_{q}\bigl(D_d(\widetilde{M})\bigr)$ of the higher homotopy groups of $M$ and $D_d(\widetilde{M})$ are isomorphisms.
\end{lem}

\begin{proof} 
Let $\widetilde{z_{1}}\in p^{-1}(z_{1})$, let $\Gamma$ be the group of covering transformations of the covering map $\map{p}{\widetilde{M}}[M]$, and let $\gamma_{1}=\id,\gamma_{2},\ldots,\gamma_{d}$ denote the elements of $\Gamma$. We thus have a map $\map{\widetilde{\psi_{1}}}{\widetilde{M}}[F_{d}(\widetilde{M})]$ given by $\widetilde{x}\longmapsto (\gamma_{1}(\widetilde{x}),\gamma_{2}(\widetilde{x}),\ldots,\gamma_{d}(\widetilde{x}))$ that makes the following diagram commute:
\begin{equation*}
\xymatrix{%
\widetilde{M} \ar[d]_{\widetilde {\psi_1}} \ar[r]^{p} &  
\ar[d]^-{\psi_1}  M  \\
F_{d}(\widetilde{M})\ar[r]^{\alpha} & D_d(\widetilde{M}),}
\end{equation*}
where $\map{\alpha}{F_{d}(\widetilde{M})}[D_d(\widetilde{M})]$ is the usual quotient map.
From this diagram, we obtain the following commutative diagram of short exact sequences:
\begin{equation*}
\xymatrix{%
1 \ar[r] & \pi_1(\widetilde{M},\widetilde{z_{1}}) \ar[r]^{p_{\#}} \ar[d]_{\widetilde{\psi_{1}}_{\#}} & \pi_1(M,z_{1}) \ar[r]^-{q_{1}} \ar[d]_{{\psi_{1\#}}} & \Gamma \ar[r] \ar[d]^{\rho} & 1\\
1 \ar[r]& P_{d}( \widetilde{M}) \ar[r]^-{\alpha_{\#}}  & B_{d}( \widetilde{M}) \ar[r]^-{q_{2}} & \sn[d] \ar[r] & 1,}
\end{equation*}
where $q_{1},q_{2}$ are the canonical quotient homomorphisms, and the homomorphism $\rho$ is that induced by the rest of the diagram. The homomorphism $\widetilde{\psi_{1}}_{\#}$ is injective since its composition with the homomorphism 
$\map{p_{1\#}}{P_{d}( \widetilde{M})}[\pi_1(\widetilde{M})]$ induced by the projection 
$\map{p_1}{F_{d}(\widetilde{M})}[\widetilde{M}]$ onto the first coordinate is the identity.

We claim that $\rho$ is injective. To see this, let $\gamma\in \ker{\rho}$, and let $c$ be a loop in $M$ based at $z_{1}$ such that $q_{1}(\ang{c})=\gamma$. Let $\widetilde{c}$ be the lift of $c$ based at $\widetilde{z_{1}}$. We have that $q_{2}\circ \psi_{1\#}(\ang{c})=\rho\circ q_{1}(\ang{c})=\id$, so $\psi_{1\#}(\ang{c})\in \ker{q_{2}}=P_{d}( \widetilde{M})$. Further, $\psi_{1}(c)=p^{-1}(c)=\brak{\gamma_{1}(\widetilde{c}),\gamma_{2}(\widetilde{c}),\ldots,\gamma_{d}(\widetilde{c})}$. For each $i=1,\ldots,d$, $\gamma_{i}(\widetilde{c})$ is an arc between $\gamma_{i}(\widetilde{z_{1}})$ and $\gamma_{j}(\widetilde{z_{1}})$ for some $j\in \brak{1,\ldots,d}$. But $\psi_{1\#}(\ang{c})\in P_{d}( \widetilde{M})$, so $j=i$, and thus each $\gamma_{i}(\widetilde{c})$ is a loop based at $\gamma_{i}(\widetilde{z_{1}})$. In particular $\ang{\widetilde{c}}\in \pi_1(\widetilde{M},\widetilde{z_{1}})$, and $\ang{c}=p_{\#}(\ang{\widetilde{c}})$, so $\gamma=q_{1}(\ang{c})=q_{1}\circ p_{\#}(\ang{\widetilde{c}})=\id$. Hence $\rho$ is injective, as claimed. The first part of the lemma then follows from the Short $5$-Lemma~\cite[Lemma 3.1, page 14]{Mac}.

For the second part, if $M$ is different from $\rp$ and $\St$ then it and its configuration spaces are finite-dimensional $CW$-complexes that are Eilenberg-MacLane spaces of type $K(\pi,1)$, and so the result is clearly true. If $M=\St$ then since $p$ is a covering map, $d=1$, $\widetilde{M}=\St$, and again the result follows directly. 

So suppose that $M=\rp$. By the Riemann-Hurwitz formula, we either have $d=1$ and $\widetilde{M}=\rp$, in which case the result is again clear, or else $d=2$ and $\widetilde{M}=\St$. From the first part, the homomorphism $\map{\psi_{1\#}}{\pi_{1}(\rp)}[B_{2}(\St)]$ is injective, and since $\pi_{1}(\rp)\cong B_{2}(\St)\cong \Z_{2}$, it is an isomorphism. The projection $\map{p_{1}}{F_2(\St)}[\St]$ onto the first coordinate is a fibration whose fibre over a point $x_{0}\in \St$ is $\St \setminus \brak{x_0}$ which is contractible, and hence it induces isomorphisms of the homotopy groups of $F_2(\St)$ and $\St$. Moreover, it is a homotopy equivalence: the map $\map{\widetilde{\psi_{1}}}{\St}[F_2(\St)]$ is a homotopy inverse of $p_{1}$. It remains to show that $D_2(\St)$ has the homotopy type of $\rp$. Observe that $\widetilde{\psi_{1}}$ is a $\Z_2$-equivariant map with respect to the free action on $\St$ given by the antipodal map and the free action on $F_2(\St) \to F_2(\St)$ given by the map that exchanges coordinates. The induced map on the quotient is $\map{\psi_{1}}{\rp}[D_2(\St)]$, and since $\widetilde{\psi_{1}}$ induces isomorphisms of the homotopy groups of $\St$ and $F_2(\St)$ and the stated actions are free, $\psi_{1}$ induces isomorphisms of the homotopy groups of $\rp$ and $D_2(\St)$. This completes the proof of the lemma.
\end{proof}

Let
\begin{equation*}
D_{d,\ldots,d}(\widetilde{M})=F_{dn}(\widetilde{M})/(\sn[d]\times\cdots \times\sn[d]),
\end{equation*}
where for $i=1,\ldots,n$, the $i\up{th}$ copy of $S_{d}$ is the symmetric group on the letters $i,i+n,\ldots,i+(d-1)n$, and let
\begin{equation*}
B_{d,\ldots,d}(\widetilde{M})=\pi_{1}\bigl(D_{d,\ldots,d}(\widetilde{M})\bigr).
\end{equation*}
Define the map $\map{\widetilde{\psi_n}}{F_n(M)}[D_{d,\ldots,d}(\widetilde{M})]$ by 
\begin{equation}\label{eq:defpsi}
\widetilde{\psi_n} \bigl((x_1,\ldots,x_n) \bigr)= \bigl(p^{-1}(\brak{x_1}),\ldots,p^{-1}(\brak{x_n}) \bigr).
\end{equation}
In a similar manner to $\psi_{n}$, $\widetilde{\psi_{n}}$ induces a homomorphism $\map{\widetilde{\psi_{n}}_{\#}}{P_n(M)}[B_{d,\ldots,d}(\widetilde{M})]$ which is the restriction of $\psi_{n\#}$ to $P_n(M)$.

\begin{prop}\label{prop:restpn}
Let $n\geq 1$. The homomorphism $\map{\widetilde{\psi_n}_{\#}}{P_n(M)}[B_{d,\ldots,d}(\widetilde{M})]$
is injective.
\end{prop}

\begin{proof}
If $d=1$ there is nothing to prove. So suppose that $d\geq 2$, in which case $M\neq \St$. The proof is by induction on $n$. For $n=1$, $\widetilde{\psi_1}=\psi_1$, and the result follows from \relem{n1inj}. The case $n=2$ is special and will be treated separately at the end of the proof. Suppose then that the result is true for some integer $n\geq 2$. Consider the commutative diagram of fibrations:
\begin{equation}\label{eq:commfib}
\begin{xy}*!C\xybox{%
\xymatrix{%
F_1(M \setminus\brak{x_1,\ldots,x_n}) \ar[d]_{\widetilde {\psi_{n+1}}\left\lvert_{M\setminus \brak{x_1,\ldots,x_n}}\right.} \ar[r] & F_{n+1}(M) \ar[d]_{\widetilde {\psi_{n+1}}} \ar[r]^{p} & F_{n}(M) \ar[d]_{\widetilde {\psi_{n}}} \\
D_d \bigl(\widetilde{M}\setminus p^{-1}(\brak{x_1,\ldots ,x_n}) \bigr) \ar[r] & D_{d,\ldots,d,d}(\widetilde{M}) \ar[r]^{\widetilde{p}} & D_{d,\ldots,d}(\widetilde{M}),}}
\end{xy} 
\end{equation}
where $\map{p}{F_{n+1}(M)}[F_{n}(M)]$ is the map that forgets the last point, 
\begin{equation*}
\map{\widetilde{p}}{D_{d,\ldots,d,d}(M)}[D_{d,\ldots,d}(M)]
\end{equation*}
is the map that forgets the last $d$ points. The long exact sequences in homotopy of these fibrations yield the following commutative diagram:  
\begin{equation*}
\xymatrix{ 
1 \ar[r] & \pi_1\left(M\setminus \brak{x_1,\ldots,x_n}\right) \ar[r] \ar[d]_{\widetilde {\psi_{n+1}}_{\#}\left\lvert_{\pi_1\left(M\setminus \brak{x_1,\ldots,x_n}\right)}\right.} & P_{n+1}(M)
\ar[r]^{p_{\#}} \ar[d]_{\widetilde {\psi_{n+1}}_{\#}} & P_n(M)  \ar[r] \ar[d]_{\widetilde {\psi_{n}}_{\#}} & 1\\
1 \ar[r] & B_d \bigl(\widetilde{M}\setminus p^{-1}\left(\brak{x_1,\ldots,x_n}\right) \bigr) \ar[r] & B_{d,\ldots,d,d}(\widetilde{M}) \ar[r]^{\widetilde{p}_{\#}} &  B_{d,\ldots,d}(\widetilde{M}) \ar[r] & 1.}
\end{equation*}
The fact that $M\neq \St$ implies the exactness of the first row. Noting that the $d$-fold covering $\map{p}{\widetilde{M}}[M]$ induces a $d$-fold covering map
\begin{equation*}
\map{\widehat{p}}{\widetilde{M}\setminus \brak{p^{-1}\left(\brak{x_1, \ldots, x_n}\right)}}[M\setminus \brak{x_1,\ldots,x_n}],
\end{equation*}
and that $\widetilde {\psi_{n+1}}\left\lvert_{\pi_1\left(M\setminus \brak{x_1,\ldots,x_n}\right)}\right.$ is equal to the map
\begin{equation*}
\map{\psi_{1}}{F_1(M \setminus\brak{x_1,\ldots,x_n})}[D_d \bigl(\widetilde{M}\setminus p^{-1}(\brak{x_1,\ldots ,x_n}) \bigr)]
\end{equation*}
for the space $F_1(M \setminus\brak{x_1,\ldots,x_n})$, it follows by applying \relem{n1inj} that the homomorphism $\widetilde {\psi_{n+1}}_{\#}\left\lvert_{\pi_1\left(M\setminus \brak{x_1,\ldots,x_n}\right)}\right.$ is injective. The injectivity of $\widetilde {\psi_{n+1}}_{\#}$ is then a consequence of the $5$-Lemma and the induction hypothesis.

It remains to prove the result in the case $n=2$. Since $d\geq 2$, observe that $M$ is not simply connected. So if $M\neq \rp$, $\pi_2(M)=\pi_2(F_{1}(M))=\brak{1}$, and the proof for the case $n=2$ follows from the case $n=1$ using exactly the same induction argument as in the previous paragraph. So suppose that $M=\rp$. In this case, we have the same commutative diagram~\reqref{commfib}, but the long exact sequence in homotopy of the fibrations yields the following commutative diagram:
\begin{equation*}
\xymatrix{ 
\pi_2(\rp)\cong \Z  \ar[r] \ar[d] & \pi_1(\rp \setminus\{x_1\})\cong \Z \ar[r] \ar[d]_{\widetilde {\psi_{1}}_{\#}} & P_{2}(\rp)\cong \mathcal{Q}_8
\ar[r] \ar[d]_{\widetilde {\psi_{2}}_{\#}} & P_1(\rp)\cong \Z_2  \ar[r] \ar[d]_{\widetilde {\psi_{1}}_{\#}} & 1\\
\pi_2(D_2(\St))\cong \Z \ar[r] & B_2( \St \setminus \{p^{-1}(x_1)\}) \ar[r] & B_{2,2}( \St) \ar[r] &  B_{2}(\St)\cong \Z_2 \ar[r] & 1.}
\end{equation*}
The first vertical arrow is the homomorphism induced on the $\pi_{2}$-level by the map $\widetilde{\psi_{1}}$.
The first (resp.\ second) vertical homomorphism is an isomorphism (resp.\ is injective) by \relem{n1inj}. The fourth vertical homomorphism is an isomorphism using the injectivity of $\widetilde {\psi_{1}}_{\#}= {\psi_{1}}_{\#}$ given by \relem{n1inj} and the fact that $P_1(\rp)\cong B_{2}(\St)\cong \Z_2$.
It follows from the strong $4$-Lemma~\cite[Lemma 3.2, page 14]{Mac} that $\widetilde {\psi_{2}}_{\#}$ is injective, and the result follows. 
\end{proof}

Now we come to the main result of this section which immediately implies \reth{covering}.
\begin{thm}
The homomorphism $\map{\psi_{n\#}}{B_n(M)}[B_{dn}(\widetilde{M})]$ is injective.
\end{thm}

\begin{proof} 
Let $x\in \ker{\psi_{n\#}}$. We take the base points $z_{1},\ldots,z_{n}$ of the braid strings to be contained within a small disc $\mathcal{D}$ lying in the interior of $M$. The inclusion of $\mathcal{D}$ in $M$ induces a homomorphism $B_{n}(\mathcal{D})\to B_{n}(M)$. Let $\map{\pi}{B_{n}(M)}[\sn]$ denote the natural braid permutation homomorphism. A generating set of $B_{n}(M)$ may be obtained by adding a set of `surface generators' (whose representatives are loops based at $z_{1}$) to a set of standard generators $\sigma_{1},\ldots,\sigma_{n-1}$ of $B_{n}(\mathcal{D})$ via the short exact sequence $1\to P_n(M) \to B_n(M) \stackrel{\pi}{\to} \sn \to 1$. Let $C$ be a set of $n!$ permutation braids, in other words a set of coset representatives of $B_{n}(M)$ modulo $P_{n}(M)$. The elements of $C$ may be chosen to belong to $B_{n}(\mathcal{D})$, and we may take the representative of the identity permutation to be the identity braid. Thus there exist $y\in P_{n}(M)$ and $w\in C$ such that $x=y\ldotp w$. Now $x\in \ker{\psi_{n\#}}$, and we have that $\psi_{n\#}(y)=\psi_{n\#}(w^{-1})$. But $w\in B_{n}(\mathcal{D})$, so the permutation $\pi(\psi_{n\#}(w^{-1}))$ of $\psi_{n\#}(w^{-1})$ permutes the elements $jn+1,\ldots,jn+n$ for each $j=0,\ldots, d-1$, while the permutation $\pi(\psi_{n\#}(y))$ of $\psi_{n\#}(y)$ permutes the elements $i,i+n,\dots, i+(d-1)n$ for each $i=1,\ldots, n$ by \req{defpsi}. Since $\pi(\psi_{n\#}(w^{-1}))=\pi(\psi_{n\#}(y))$, it follows that this permutation must be the identity, hence $w=\id$, and thus $x=y\in P_{n}(M)$. By \repr{restpn}, the restriction $\widetilde{\psi_{n}}_{\#}$ of $\psi_{n\#}$ to $P_n(M)$ is injective, and we conclude that $x=\id$, hence $\psi_{n\#}$ is injective.
\end{proof}

\section{The classification of the finite subgroups of $B_{n}(\rp)$ and $\operatorname{\mathcal{MCG}}(\rp,n)$}\label{sec:classfinite}

In this section, we prove Theorems~\ref{th:finitebnrp2} and~\ref{th:mcg}. Since $B_{n}(\rp)$ and $\operatorname{\mathcal{MCG}}(\rp,n)$ are finite for $n\in \brak{1,2}$, we shall suppose in what follows that $n\geq 3$. The proof of \reth{finitebnrp2} is divided into two parts:
\begin{enumerate}[(a)]
\item\label{it:proofa} in \resec{proofa}, using \reth{covering} in the case $M=\rp$, and applying \reth{finitebn}, we obtain necessary conditions for a finite group to be realised as a subgroup of $B_{n}(\rp)$. This enables us to establish a list in \repr{neccond} of the finite groups that are candidates to be subgroups of $B_{n}(\rp)$.
\item in \resec{proofb}, we show in \repr{realise} that the candidates of \repr{neccond} are indeed realised as subgroups of $B_{n}(\rp)$ using \req{sesrp2} and geometric constructions similar to those of~\cite{GG6} for $B_{n}(\St)$.
\end{enumerate}
The proof of \reth{mcg} is given in \resec{mcg}.

\subsection{Necessary conditions for a finite group to be a subgroup of $B_{n}(\rp)$}\label{sec:proofa}

In this section, we prove the following necessary condition for a finite subgroup $H$ to be realised as a subgroup of $B_{n}(\rp)$.
\begin{prop}\label{prop:neccond}
Let $n\geq 3$, and let $H$ be a finite subgroup of $B_{n}(\rp)$. Then $H$ is isomorphic to a subgroup of one of the following groups: $\dic{8n}$, $\dic{8(n-1)}$, $\oonestar$ if $n\equiv 0,1\pmod 3$, or $\istar$ if $n\equiv 0,1,6,10\pmod{15}$.
\end{prop}

\begin{proof}
Take $M=\rp$, $\widetilde{M}=\St$ and $d=2$ in the statement of \reth{covering}. Let $m$ denote the order of $H$. With the notation of \resec{inject}, $\psi_{n\#}(H)$ is a subgroup of $B_{2n}(\St)$ isomorphic to $H$. In particular, $\psi_{n\#}(H)$ is of order $m$. Applying \reth{finitebn}, it follows that $H$ is isomorphic to a subgroup of one of the following maximal 
finite groups of $B_{2n}(\St)$:
\begin{enumerate}
\item\label{it:S2nfinitea} $\Z_{2(2n-1)}$.
\item $\dic{8n}$.
\item $\dic{8(n-1)}$.
\item\label{it:S2nfinited} $\tonestar$ if $2n\equiv 4 \pmod 6$, \emph{i.e.} $n\equiv 2\pmod 3$.
\item $\oonestar$ if $2n\equiv 0,2 \pmod 6$, \emph{i.e.} $n\equiv 0,1\pmod 3$.
\item $\istar$ if $2n\equiv 0,2,12,20\pmod{30}$, \emph{i.e.} $n\equiv 0,1,6,10\pmod{15}$.
\end{enumerate}
The fact that $n\geq 3$ implies that $\Z_{2(2n-1)}$ and $\dic{8(n-1)}$ are maximal in $B_{2n}(\St)$.

Case~(\ref{it:S2nfinitea}) may be deleted from the list. To see this, note that if $\psi_{n\#}(H)$ is a subgroup of $\Z_{2(2n-1)}$ then $H$ is cyclic of order a divisor of $2(2n-1)$. On the other hand, its order $m$ must divide the torsion of $B_{n}(\rp)$, which is $4n$ and $4(n-1)$ by \repr{agt}(\ref{it:agt1}). But $\gcd (2n-1,2n)=\gcd (2n-1, 2(n-1))=1$, and so $m=1$ or $2$. Hence $H$ is either trivial or equal to $\ang{\ft}$ (by \repr{agt}(\ref{it:agt2})), and both are contained in $\dic{8n}$, for example.

Case~(\ref{it:S2nfinited}) may be also be deleted from the list. To see this, assume that $n\equiv 2 \pmod 3$. Suppose that $H$ is a subgroup of $\tonestar\cong \quat \rtimes \Z_{3}$ not contained in the $\quat$-factor. Then $H$ contains elements of order $3$, and hence $3$ divides the torsion of $B_{n}(\rp)$. Once more, by \repr{agt}(\ref{it:agt1}), $3$ divides $n$ or $n-1$, which contradicts the fact that $n\equiv 2 \pmod 3$.
Hence $H$ is isomorphic to a subgroup of $\quat$, which is realised as a subgroup of $\dic{8n}$, for example. This completes the proof of the proposition.
\end{proof}

\subsection{The realisation of the finite subgroups of $B_{n}(\rp)$}\label{sec:proofb}

In this section, we prove that the groups listed in \repr{neccond} are indeed realised as subgroups of $B_{n}(\rp)$.
\begin{prop}\label{prop:realise}
Let $n\geq 3$, and let $H$ be one of the following groups: $\dic{8n}$; $\dic{8(n-1)}$; $\oonestar$ if $n\equiv 0,1\pmod 3$; or $\istar$ if $n\equiv 0,1,6,10\pmod{15}$. Then $B_{n}(\rp)$ contains an isomorphic copy of $H$.
\end{prop}

\begin{proof}
Let $X\subset \rp$ be an $n$-point subset, let $
\left\{ \begin{aligned}
p \colon\thinspace \St &\to \rp\\
x & \longmapsto [x]
\end{aligned}\right.$ denote the natural projection, and let $\widetilde{X}=p^{-1}(X)$. Let $\operatorname{\mathcal{H}}(\rp,X)$ (resp.\ $\operatorname{\mathcal{H}}(\St, \widetilde{X})$) denote the group of homeomorphisms of $\rp$ (resp.\ $\St$) that leave $X$ (resp.\ $\widetilde{X}$) invariant, and let 
$\operatorname{\mathcal{H}^+}(\St, \widetilde{X})$) denote the index $2$ subgroup of $\operatorname{\mathcal{H}}(\St, \widetilde{X})$ consisting of orientation-preserving homeomorphisms.
Given $H$ as in the statement of the proposition, let $G$ be the quotient of $H$ by its (normal) subgroup of order $2$. Then
\begin{equation}\label{eq:quotient}
G\cong 
\begin{cases}
\dih{4(n-i)} & \text{if $H=\dic{8(n-i)}$, where $i=0,1$}\\
\sn[4] & \text{if $H=\oonestar$}\\
\an[5] & \text{if $H=\istar$.}
\end{cases}
\end{equation}
Note that $G$ consists of rotations and may be realised as a subgroup of $\operatorname{\mathcal{H}^+}(\St,\widetilde{X})$ (\emph{cf.} the solution of the Nielsen realisation problem~\cite{K}). We may assume that the points of $\widetilde{X}$ are symmetrically arranged with respect to the regular polyhedron associated with $G$. 

 

The antipodal map $\map{\tau}{\St}[\St]$, defined by $\tau(\widetilde{x})=-\widetilde{x}$ for all $\widetilde{x}\in \St$, belongs to $\operatorname{\mathcal{H}}(\St, \widetilde{X}) \setminus \operatorname{\mathcal{H}^+}(\St, \widetilde{X})$, and commutes with any rotation. In particular, $\tau$ commutes with all of the elements of $G$, and so $G$ is a subgroup of 
\begin{equation*}
\Gamma=\setl{\gamma \in \operatorname{\mathcal{H}^+}(\St,\widetilde{X})}{\gamma \circ \tau=\tau \circ \gamma}.
\end{equation*}
For each $\gamma\in\Gamma$, the map $\widehat{\gamma}$ defined by $\widehat{\gamma}(y)=p(\gamma(x))$, where $y\in \rp$ and $x\in p^{-1}(\brak{y})$, is a well-defined homeomorphism of $\rp$ that leaves $X$ invariant. The correspondence $\gamma \longmapsto \widehat{\gamma}$ defines an injective group homomorphism $\map{\psi}{\Gamma}[\operatorname{\mathcal{H}}(\rp,X)]$. The injectivity of $\psi$ follows from that fact that there are exactly two elements, $\id_{\St}$ and $\tau$, of  $\operatorname{\mathcal{H}}(\St,\widetilde{X})$ that cover $\id_{\rp}$, but $\tau\notin \operatorname{\mathcal{H}^+}(\St, \widetilde{X})$. 


It follows that $\psi(G)\cong G$, and thus $\operatorname{\mathcal{H}}(\rp,X)$ contains an isomorphic copy of $G$. Let $\operatorname{\mathcal{MCG}}(\rp,n)$ denote the mapping class group of the projective plane relative to $X$. There is a natural homomorphism $\map{\varphi}{\operatorname{\mathcal{H}}(\rp,X)} [\operatorname{\mathcal{MCG}}(\rp,n)]$ given by associating the isotopy class of the homeomorphism relative to $X$. The restriction of this homomorphism to $\psi(G)$ is injective. To see this, let $f\in \psi(G)$ be such that $\varphi(f)$ is the trivial mapping class, and let $g\in G$ be such that $\psi(g)=f$. Then there exists an isotopy $\brak{f_{t}}_{t\in [0,1]}$ from $\id_{\rp}$ to $f$ relative to $X$. This isotopy lifts to an isotopy $\brak{\widetilde{f}_{t}}_{t\in [0,1]}$ from $\id_{\St}$ to some lift $\widetilde{f}\in \operatorname{\mathcal{H}^+}(\St, \widetilde{X})$ of $f$ relative to $\widetilde{X}$. By construction, $\widetilde{f}\in \Gamma$, and since $\widetilde{f}$ and $g$ are both sent to $f$ by $\psi$, the injectivity of $\psi$ implies that $\widetilde{f}=g$. Now the isotopy $\brak{\widetilde{f}_{t}}_{t\in [0,1]}$ is relative to $\widetilde{X}$, so $g$ is a finite-order homeomorphism that fixes the points of $\widetilde{X}$ (of which there are at least $6$), and thus $g=\id_{\St}$ by~\cite{E,vK}, and $f=\id_{\rp}$. This proves the injectivity of $\varphi\left\lvert_{\psi(G)}\right.$.  

It follows that $\operatorname{\mathcal{MCG}}(\rp,n)$ contains an isomorphic copy of $G$. Finally, the short exact sequence~\reqref{sesrp2} is the analogue for $\rp$ of the short exact sequence~(1-1) of \cite[page 760]{GG6}. It is derived as for $\St$ (see equation~(1-2) of~\cite[page 760]{GG6}), by taking the long exact sequence in homotopy of the fibration $\operatorname{\mathcal{H}}(\rp)\to D_{n}(\rp)$ defined by $f\longmapsto f(X)$, and whose fibre over $X$ is $\operatorname{\mathcal{H}}(\rp,X)$. As in equation~(1-2), the homomorphism $B_{n}(\rp) \to\operatorname{\mathcal{MCG}}(\rp,n)$ of~\reqref{sesrp2} is a boundary operator. Following the proof of \cite[Theorem 1.3, pages 764--765]{GG6}, we see that $G$ lifts to an isomorphic copy of $H$ lying in $B_{n}(\rp)$. This completes the proof of the proposition, and together with \repr{neccond}, that of \reth{finitebnrp2}.
\end{proof}

\subsection{The classification of the finite subgroups of $\operatorname{\mathcal{MCG}}(\rp,n)$}
\label{sec:mcg}

We close \resec{classfinite} by proving \reth{mcg}.
\begin{proof}[Proof of \reth{mcg}.]
As in the case of $\St$~\cite[Remarks~2.1]{GG6}, the short exact sequence~\reqref{sesrp2} induces a bijection between the maximal finite subgroups of $B_{n}(\rp)$ and those of $\operatorname{\mathcal{MCG}}(\rp,n)$. The result then follows from \reth{finitebnrp2} and \req{quotient}.
\end{proof}

\section{The algebraic realisation of the finite dicyclic subgroups of
$B_{n}(\rp)$}\label{sec:realdicyc}

In this section, we show in \repr{realdic} that $\dic{8n}$ and $\dic{8(n-1)}$ are realised as
subgroups of $B_{n}(\rp)$ for all $n\geq 2$ by giving explicit algebraic realisations. We first recall Van Buskirk's presentation of $B_n(\rp)$.
\begin{prop}[Van Buskirk~\cite{VB}]\label{prop:present}
The following constitutes a presentation of the group $B_n(\rp)$:\\
\underline{\textbf{generators:}}
$\si{1},\ldots,\si{n-1},\rh{1},\ldots,\rh{n}$.\\
\underline{\textbf{relations:}}
\begin{align}
\si{i}\si{j} &=\si{j}\si{i}\quad\text{if $\lvert i-j\rvert\geq 2$,}\notag\\
\si{i}\si{i+1}\si{i}&=\si{i+1}\si{i}\si{i+1} \quad\text{for $1\leq i\leq n-2$,}\notag\\
\si{i}\rh{j}&=\rh{j}\si{i}\quad\text{for $j\neq i,i+1$,}\label{eq:sirj}\\
\rh{i+1}&=\sii{i}\rh{i}\sii{i} \quad\text{for $1\leq i\leq n-1$,}\label{eq:sirisi}\\
\rh[-1]{i+1}\rh[-1]{i}\rh{i+1}\rh{i}&= \si[2]{i} \quad\text{for $1\leq i\leq n-1$,}\notag\\
\rh[2]{1}&=\si{1}\si{2}\cdots\si{n-2}\si[2]{n-1} \si{n-2}\cdots\si{2}\si{1}.\notag
\end{align}
\end{prop}

In terms of this presentation, by~\cite[equations (4) and (7), pages 770--771]{GG2}, we have the following useful identities:
\begin{align}
\rh{j}&=\ssni{j-1}{1}\rh{1}\ssni{1}{j-1} \quad\text{for $j=1,\ldots, n$}\label{eq:rjr1}\\
\rho_{n}^{-2} &=\si{n-1}\cdots \si{2}\si[2]{1} \si{2}\cdots \si{n-1}, \label{eq:rn2}
\end{align}
as well as elements of $B_n(\rp)$ of order $4n$ and of order $4(n-1)$ respectively~\cite[Proposition 26]{GG2}:
\begin{align*}
a &= \sii{n-1}\cdots\sii{1}\cdot\rh{1}\\
b &= \sii{n-2}\cdots\sii{1}\cdot\rh{1}
\end{align*}
that by \cite[Remark 27]{GG2} satisfy
\begin{equation}\label{eq:powerab}
a^n=\rho_{n}\cdots \rho_{1} \quad\text{and}\quad b^{n-1}=\rho_{n-1}\cdots \rho_{1}.
\end{equation}

\begin{rem}\label{rem:permute}
From~\cite[pages 777--778]{GG2}, we have:
\begin{enumerate}[(a)]
\item\label{it:permsi} conjugation by $a^{-1}$ permutes cyclically the following elements:
\begin{equation*}
\si{1},\ldots, \si{n-1}, a^{-1}\si{n-1}a, \sii{1},\ldots, \sii{n-1}, a^{-1}\sii{n-1}a.
\end{equation*}
\item\label{it:permrhi} conjugation by $a^{-1}$ permutes cyclically the following elements:
\begin{equation*}
\rh{1},\ldots \rh{n}, \rhi{1},\ldots, \rhi{n}.
\end{equation*}
\item conjugation by $b^{-1}$ permutes cyclically the following elements:
\begin{equation*}
\si{1},\ldots, \si{n-2}, b^{-1}\si{n-2}b, \sii{1},\ldots, \sii{n-2}, b^{-1}\sii{n-1}b.
\end{equation*}
Note that there is a typographical error in line~16 of \cite[page~778]{GG2}: it should
read `\ldots shows that $b^{-2}\si{n-2}b^2=\sii{1}$ \ldots', and not `\ldots shows that
$b^{-2}\si{n-1}b^2=\sii{1}$ \ldots'.
\end{enumerate}
\end{rem}

Let $n\geq 2$. From \repr{agt}(\ref{it:agt1}), the maximal finite cyclic subgroups of $B_{n}(\rp)$ are $\Z_{4n}$ and $\Z_{4(n-1)}$. Considered as an element of $B_{n}$, the half twist braid $\garside$ satisfies the relation:
\begin{equation}\label{eq:propgarside}
\garside^{-1} \sigma_{i} \garside=\sigma_{n-i} \quad\text{for all $i=1,\ldots,n-1$,}
\end{equation}
and since there is a group homomorphism $B_{n}\to B_{n}(\rp)$ induced by inclusion of a topological disc in $\rp$, the same relation holds in $B_{n}(\rp)$. We then have the following algebraic realisations of the dicyclic subgroups given by \reth{finitebnrp2}:
\begin{prop}\label{prop:realdic}
Let $n\geq 2$. Then:
\begin{enumerate}[(a)]
\item $\ang{a,\garside}\cong \dic{8n}$.
\item $\ang{b,\garside a^{-1}}\cong \dic{8(n-1)}$.
\end{enumerate}
\end{prop}

Before proving \repr{realdic}, let us state and prove the following useful lemma.
\begin{lem}\label{lem:conjri}
For all $1\leq i\leq n$, $\garside^{-1} \rho_{i} \garside=\rho_{n+1-i}^{-1}$.
\end{lem}

\begin{proof}[Proof of \relem{conjri}.]
We prove the result by induction on $i$. First let $i=1$. Using \req{garside}, we have 
\begin{align*}
\garside^{-1} \rho_{1} \garside=& 
(\sigma_{1})^{-1} (\sigma_{1}\sigma_{2})^{-1} \cdots (\sigma_{1}\cdots \sigma_{n-2})^{-1}
\sigma_{n-1}^{-1}\cdots \sigma_{1}^{-1} \cdot\rho_{1} \cdot \sigma_{1}^{-1}\cdots \sigma_{n-1}^{-1}
\cdot \sigma_{n-1}\cdots \sigma_{1} \cdot\\
& (\sigma_{1}\cdots \sigma_{n-1}) (\sigma_{1}\cdots
\sigma_{n-2}) \cdots (\sigma_{1}\sigma_{2})(\sigma_{1})\\
=& (\sigma_{1})^{-1} (\sigma_{1}\sigma_{2})^{-1} \cdots (\sigma_{1}\cdots \sigma_{n-2})^{-1}
\rho_{n}
\cdot \rho_{n}^{-2}\cdot\\
& (\sigma_{1}\cdots
\sigma_{n-2}) \cdots (\sigma_{1}\sigma_{2})(\sigma_{1}) \quad\text{by equations~\reqref{rjr1}
and~\reqref{rn2}}\\
=& \rho_{n}^{-1} \quad\text{by \req{sirj}.}
\end{align*}
Now suppose that the result is true for $1\leq i\leq n-1$. By \req{sirisi} we have 
\begin{equation}\label{eq:rhii}
\rho_{i}^{-1}= \sigma_{i}^{-1} \rho_{i+1}^{-1} \sigma_{i}^{-1}.
\end{equation}
So
\begin{align*}
\garside^{-1} \rho_{i+1} \garside&= \garside^{-1} \sigma_{i}^{-1} \rho_{i} \sigma_{i}^{-1}
\garside \quad\text{by \req{sirisi}}\\
&= \sigma_{n-i}^{-1} \rho_{(n-i)+1}^{-1} \sigma_{n-i}^{-1} \quad\text{by induction and \req{propgarside}}\\
&= \rho_{n-i}^{-1} \quad\text{by \req{rhii}.}
\end{align*}
The result follows by induction.
\end{proof}

\begin{proof}[Proof of \repr{realdic}.]\mbox{}
\begin{enumerate}[(a)]
\item\label{it:dic8n} We know that $a$ is of order $4n$, $\garside$ is of order $4$, and that $a^{2n}=\garside^2=\ft$ by \repr{agt}(\ref{it:agt2}). By \relem{conjri}, we obtain
\begin{align*}
\garside a \garside^{-1}&= \garside \sigma_{n-1}^{-1}\cdots \sigma_{1}^{-1} \rho_{1}
\garside^{-1}\\
&= \sigma_{1}^{-1}\cdots \sigma_{n-1}^{-1} \rho_{n}^{-1}= (\rho_{n}\sigma_{n-1} \cdots
\sigma_{1})^{-1}=a^{-1} \quad\text{by \req{rjr1}.}
\end{align*}
This proves that the subgroup $\ang{a,\garside}$ of $B_{n}(\rp)$ is isomorphic to a quotient of $\dic{8n}$. Now $\garside$ is a non-pure braid of order $4$, and the elements of $\ang{a}$ of order $4$ are $a^{\pm n}$ which are pure braids by \req{powerab}. This implies that $\ang{a}\cap \ang{\garside}=\varnothing$, thus $\ang{a,\garside}$ contains at least $4n+1$ distinct elements, and so is isomorphic to $\dic{8n}$.

\item We know that $b$ is of order $4(n-1)$. Moreover, since $\ang{a,\garside}\cong \dic{8n}$ and $a$ generates its subgroup of order $4n$, it follows from standard properties of the dicyclic group that $\garside a^{-1}$ is of order~$4$. 
Again by \repr{agt}(\ref{it:agt2}), $b^{2(n-1)}=(\garside a^{-1})^2=\ft$. Further,
\begin{align*}
\garside a^{-1} \cdot b \cdot a \garside^{-1}&= \garside a^{-1} \cdot
\sii{n-2}\cdots\sii{1}\cdot\rh{1} \cdot a \garside^{-1}\\
&= \garside \sii{n-1}\cdots\sii{2}\rh{2} \garside^{-1} \quad\text{by Remarks~\ref{rem:permute}(\ref{it:permsi})
and~(\ref{it:permrhi})}\\
&= \sii{1}\cdots\sii{n-2}\rhi{n-1} \quad\text{by \relem{conjri}}\\
&= (\rho_{n-1} \sigma_{n-2} \cdots \sigma_{1})^{-1}=b^{-1} \quad\text{by \req{rjr1}}.
\end{align*}
This proves that the subgroup of $B_{n}(\rp)$ generated by $b$ and $\garside a^{-1}$ is isomorphic
to a quotient of the dicyclic group of order $8(n-1)$. An argument similar to that of~(\ref{it:dic8n}) above shows 
that $\ang{b,\garside a^{-1}}$ contains at least $4(n-1)+1$ distinct elements, and so is isomorphic to $\dic{8(n-1)}$.\qedhere
\end{enumerate}
\end{proof}

\section{The linearity of braid groups of low-genus surfaces}\label{sec:linear}

In this section, we prove \reth{linear}. Using the theorem due to Malcev~\cite{Mal,W} that a group is linear if it contains a finite-index linear subgroup, we first give a short proof of the linearity of $B_{n}(\St)$ different from those in~\cite{Bar,BB}.
\begin{prop}[\cite{Bar,BB}]\label{prop:linear}
Let $n\in \N$. Then $B_{n}(\St)$ is linear.
\end{prop}

\begin{proof}
If $n\leq 3$ then $B_{n}(\St)$ is finite, and so is linear. Assume then that $n\geq 4$. Since $P_{n}(\St)\cong P_{n-3}(\St\setminus \brak{x_{1},x_{2},x_{3}}) \times \Z_{2}$ by~\cite[Theorem 4]{GG1} and $P_{n}(\St)$ is of finite index in $B_{n}(\St)$, it suffices to prove the linearity of $P_{n-3}(\St\setminus \brak{x_{1},x_{2},x_{3}})$. Now by~\cite[Proposition 2.5]{GG8}, $P_{n-3}(\St\setminus \brak{x_{1},x_{2},x_{3}})$ is isomorphic to $P_{n-3}(\dt\setminus \brak{x_{1},x_{2}})$, which in turn is isomorphic to a subgroup of $P_{n-1}$. The linearity of $B_{n-1}$ implies that of $P_{n-3}(\St\setminus \brak{x_{1},x_{2},x_{3}})$ as required.
%
\end{proof}

\begin{proof}[Proof of \reth{linear}.]\mbox{}
\begin{enumerate}[(a)]
\item Let $n\in \N$, and let $M$ be a compact, connected surface. First, suppose that $M$ is orientable of genus zero. Then $M$ is homeomorphic to $\St$ with a finite number, $r\geq 0$ say, of disjoint open discs removed. If $r=0$ then we are in the case of $\St$ which follows from~\cite{Bar,BB} or from \repr{linear}, while if $r\geq 1$, $B_{n}(M)$ is isomorphic to the $n\up{th}$ braid group of the disc with $(r-1)$ discs removed. By~\cite{GG8}, it is isomorphic a subgroup of $B_{n+r-1}$, so is linear. Now suppose that $M$ is non-orientable of genus $1$. Then it is homeomorphic to $\rp$ with a finite number of disjoint open discs removed. By \reco{embed}, $B_{n}(M)$ embeds in $B_{2n}(\widetilde{M})$, where $\widetilde{M}$ is the orientable double covering of $M$. But $\widetilde{M}$ is of genus zero, and the result follows from the previous case.
\item Let $x\in \mathbb{T}^2$ and let $n\in \N$. By~\cite[Lemma 17]{BGG}, $P_{n+1}(\mathbb{T}^2)$ is isomorphic to the direct product of $\Z^2$ with $P_{n}(\mathbb{T}^2\setminus\brak{x})$. The linearity of $\Z^2$, the fact that the $n\up{th}$ pure braid group is of finite index in the corresponding $n\up{th}$ braid group and Malcev's theorem then imply the first statement. The second statement is a consequence of the first, noting that $P_{1}(\mathbb{T}^2\setminus\brak{x})$ is a free group of rank $2$.\qedhere
\end{enumerate}
\end{proof}


\begin{thebibliography}{BGG}
{\small


\bibitem[AM]{AM} A.~Adem and R.~J.~Milgram, Cohomology of finite
groups, Springer-Verlag, New York-Heidelberg-Berlin (1994).

\bibitem[A1]{A1} E.~Artin, Theorie der Z\"opfe, \emph{Abh.\ Math.\ Sem.\ Univ.\ Hamburg} \textbf{4} (1925), 47--72.

\bibitem[A2]{A2} E.~Artin, Theory of braids, \emph{Ann.\ Math.} \textbf{48} (1947), 101--126.

\bibitem[Ba]{Bar} V.~Bardakov, The structure of the group of conjugating automorphisms and the linear representation of the braid groups of some manifolds, \emph{Siberian Math.\ J.} \textbf{46} (2005), 13--23.


\bibitem[BGG]{BGG} P.~Bellingeri, S.~Gervais and J.~Guaschi, Lower central series of Artin-Tits and surface braid groups, \emph{J.~Algebra} \textbf{319} (2008), 1409--1427.

\bibitem[Bi]{Big} S.~J.~Bigelow, Braid groups are linear, \emph{J.\ Amer.\ Math.\ Soc.} \textbf{14} (2001),
471--486.

\bibitem[BB]{BB} S.~J.~Bigelow and R.~D.~Budney, The mapping class group of a genus two surface is linear, \emph{Algebraic and Geometric Topology} \textbf{1} (2001), 699--708.

\bibitem[Br]{Br} K.~S.~Brown, Cohomology of groups, Graduate Texts in Mathematics \textbf{87}, Springer-Verlag, New York-Berlin, 1982.

\bibitem[BCG]{BCG} E.~Bujalance, F.~J.~Cirre and J.~M.~Gamboa, Automorphism groups of the real projective plane with holes and their conjugacy classes within its mapping class group, \emph{Math. Ann.} \textbf{332} (2005), 253--275.

\bibitem[CG]{CG} F.~R.~Cohen and S.~Gitler, On loop spaces of configuration spaces, \emph{Trans.\ Amer.\ Math.\ Soc.} \textbf{354} (2002), 1705--1748.

\bibitem[E]{E} S.~Eilenberg, Sur les transformations périodiques de la surface de la sphère, \emph{Fund.\ Math.} \textbf{22} (1934), 28--41.

\bibitem[FH]{FH} E.~Fadell and S.~Y.~Husseini,  Geometry and topology of configuration spaces, Springer Monographs in Mathematics, Springer-Verlag, Berlin, 2001.


\bibitem[FoN]{FoN} R.~H.~Fox and L.~Neuwirth, The braid groups,
\emph{Math.\ Scandinavica} \textbf{10} (1962), 119--126.

\bibitem[GG1]{GG1} D.~L.~Gon\c{c}alves and J.~Guaschi, The roots of
the full twist for surface braid groups, \emph{Math.\ Proc.\ Camb.\ Phil.\
Soc.} \textbf{137} (2004), 307--320.

\bibitem[GG2]{GG2} D.~L.~Gon\c{c}alves and J.~Guaschi, The braid groups of the projective plane,
\emph{Algebraic and Geometric Topology} \textbf{4} (2004), 757--780.


\bibitem[GG3]{GG4} D.~L.~Gon\c{c}alves and J.~Guaschi, The braid groups of the projective plane and the Fadell-Neuwirth short exact sequence, \emph{Geom.\ Dedicata} \textbf{130} (2007), 93--107.

\bibitem[GG4]{GG5} D.~L.~Gon\c{c}alves and J.~Guaschi, The quaternion group as a subgroup of the sphere braid groups, \emph{Bull.\ London Math.\ Soc.} \textbf{39} (2007), 232--234.

\bibitem[GG5]{GG6} D.~L.~Gon\c{c}alves and J.~Guaschi, The classification and the conjugacy classes of the finite subgroups of the sphere braid groups, \emph{Algebraic and Geometric Topology} \textbf{8} (2008), 757--785.


\bibitem[GG6]{GG8} D.~L.~Gon\c{c}alves and J.~Guaschi, The lower
central and derived series of the braid groups of the finitely-punctured sphere, \emph{J.~Knot Theory and its Ramifications} \textbf{18} (2009), 651--704.

\bibitem[GG7]{GG9} D.~L.~Gon\c{c}alves and J.~Guaschi, Classification of the virtually cyclic subgroups of the pure braid groups of the projective plane, \emph{J. Group Theory}, to appear.



\bibitem[Ke]{K} S.~P.~Kerckhoff, The Nielsen realization problem, \emph{Bull.\ Amer.\ Math.\ Soc.} \textbf{2} (1980), 452--454. 

\bibitem[vK]{vK} B.~von Kerékj\'art\'o, \"Uber die periodischen Transformationen der Kreisscheibe und der
Kugelfl\"ache, \emph{Math.\ Ann.} \textbf{80} (1919), 36--38.

\bibitem[Ko]{Kor} M.~Korkmaz, On the linearity of certain mapping class groups, \emph{Turkish
J.\ Math.} \textbf{24} (2000), 367--371.

\bibitem[Kr1]{Kr1} D.~Krammer, The braid group $B_4$ is linear, \emph{Invent.\ Math.} \textbf{142} (2000),
451--486.

\bibitem[Kr2]{Kr2} D.~Krammer, Braid groups are linear, \emph{Ann.\ Math.} \textbf{155} (2002), 131--156.

\bibitem[Mac]{Mac} S.~Mac~Lane, Homology, Die Grundlehren der mathematischen Wissenschaften, Bd.\ 114,
Academic Press Inc., New York, 1963.

\bibitem[Mal]{Mal} A.~I.~Malcev, On isomorphic matrix representations of infinite groups, \emph{Rec. Math. [Mat. Sbornik] N.S.} \textbf{8} (1940), 405--422.

\bibitem[Mu]{M} K.~Murasugi, Seifert fibre spaces and braid groups, \emph{Proc.\ London Math.\ Soc.}
\textbf{44} (1982), 71--84.

\bibitem[PR]{PR} L.~Paris and D.~Rolfsen, Geometric subgroups of surface braid groups,
\emph{Ann.\ Inst.\ Fourier} \textbf{49} (1999), 417--472.

\bibitem[Sc]{Sc} G.~P.~Scott, Braid groups and the group of homeomorphisms of a surface, \emph{Proc.\ Camb.\ Phil.\ Soc.} \textbf{68} (1970), 605--617.


\bibitem[W]{W} B.~A.~F.~Wehrfritz, Infinite linear groups, An account of the group-theoretic
properties of infinite groups of matrices, Ergebnisse der Matematik und ihrer Grenzgebiete, Band 76, Springer-Verlag, New York, 1973.

\bibitem[VB]{VB} J.~Van~Buskirk, Braid groups of compact
$2$-manifolds with elements of finite order, \emph{Trans.\ Amer.\
Math.\ Soc.} \textbf{122} (1966), 81--97.

\bibitem[Z]{Z} O.~Zariski, The topological discriminant group of a Riemann surface of genus $p$, 
\emph{Amer.\ J.\ Math.} \textbf{59} (1937), 335--358.

}

\end{thebibliography}
\end{document}